\documentclass[11pt]{article}
\usepackage{amsmath}
\usepackage{amsthm}
\usepackage{amssymb}
\usepackage{graphicx}
\usepackage{multirow} 
\usepackage{tabularx}
\newcommand{\Z}{\mathbb{Z}}
\newcolumntype{R}{>{\raggedleft\arraybackslash}X}
\newcolumntype{L}{>{\raggedright\arraybackslash}X}
\usepackage[a4paper,margin=3cm,right=3cm,includehead]{geometry}
\newtheorem{theorem}{Theorem}

\title{ \bf The degree-diameter problem for circulant graphs of degree 8 and 9}
\author{R.R. Lewis\\[-3pt]
\small Department of Mathematics and Statistics\\[-3pt]
\small The Open University\\[-3pt]
\small Milton Keynes, UK\\[-3pt]
\small \texttt{robert.lewis@open.ac.uk}}
\date{10th April 2014}



\makeatletter
\renewcommand\section{\@startsection {section}{1}{\z@}%
                                   {-2.5ex \@plus -1ex \@minus -.2ex}%
                                   {1.3ex \@plus.2ex}%
                                   {\normalfont\bf}}
\makeatother

\begin{document}
\maketitle

\begin{abstract}

This paper considers the degree-diameter problem for undirected circulant graphs. The focus is on extremal graphs of given (small) degree and arbitrary diameter. The published literature only covers graphs of up to degree 7. The approach used to establish the results for degree 6 and 7 has been extended successfully to degree 8 and 9. Candidate graphs are defined as functions of the diameter for both degree 8 and degree 9. They are proven to be extremal for small diameters. They establish new lower bounds for all greater diameters, and are conjectured to be extremal. The existence of the degree 8 solution is proved for all diameters. Finally some conjectures are made about solutions for circulant graphs of higher degree.
\end{abstract}


\section{Introduction}

The degree-diameter problem is to identify extremal graphs, having the largest possible number of vertices for a given maximum degree and diameter. There is an extensive body of work on the degree-diameter problem in the literature, but relatively few papers consider the specific case of undirected circulant graphs. For circulant graphs there are two main areas of research: largest graphs of given (small) degree and arbitrary diameter, and of given (small) diameter and arbitrary degree. This paper will focus on the first, where the goal for circulant graphs of any given degree, considered as Cayley graphs of cyclic groups, is to determine the order and generator sets for extremal graphs as functions of their diameter.

So far this has only been achieved for graphs of degree $d\leq4$. For degree 2 and 3, the solutions are straightforward. Chen and Jia included a proof for degree 4 in their 1993 paper \cite {Chen}, and Dougherty and Faber presented a proof for degree 5 in 2004 \cite {Dougherty}. Chen and Jia also constructed a family of graphs of even degree $d\geq4$ and all diameters $k\geq d/2$. This family establishes a useful lower bound for extremal graphs of all unresolved degrees. For both degree 6 and degree 7, Dougherty and Faber \cite {Dougherty} constructed families of graphs, defined as functions of the diameter, which improved on Chen and Jia's lower bound. These graphs have been proved by computer search to be extremal for small diameters and are conjectured to remain extremal for all greater diameters. Dougherty and Faber also proved an upper bound for Abelian Cayley graphs valid for all degrees and diameters. Since then it appears that nothing further has been published on extremal circulant graphs of given degree and arbitrary diameter. Miller and Siran's comprehensive 2013 survey of the state of the art of the degree-diameter problem \cite {Miller} includes only a relatively brief section on Abelian Cayley graphs, mostly relating to Dougherty and Faber's results for circulant graphs. Therefore the Dougherty and Faber constructions remain the best known solutions of degree 6 and 7 for arbitrary diameter. For all higher degrees no solutions have been published which improve on the Chen and Jia constructions.
In this paper we will review these upper and lower bounds and the extremal and best known solutions up to degree 7. 

The main result of this paper is the construction of families of circulant graphs of both degree 8 and 9, defined as functions of the diameter. These graphs improve on Chen and Jia's lower bounds. They have been proved by computer search to be extremal for small diameters above a threshold value.
A proof of the existence of these graphs is given for the case of degree 8 for arbitrarily large diameters. This proof closely follows the approach taken by Dougherty and Faber for their proof of the existence of the degree 6 construction \cite {Dougherty}. The full proof, covering the separate cases of even and odd diameter and handling the exceptions for all combinations of generator elements, extends to 20 pages. Therefore in this paper the details are included only for a representative subcase. The full proof is available on ArXiv \cite {Lewis}. These new families of graphs are conjectured to be extremal for all diameters above a defined threshold.


\section{Properties of circulant graphs}

A circulant graph $X(\Z_n,C)$ of order $n$ may be defined as a Cayley graph whose vertices are the elements of the cyclic group $\Z_n$ where two vertices $i,j$ are connected by an arc $(i, j)$ if and only if $j-i$ is an element of $C$, a subset of $\Z_n \setminus 0$, called the connection set. If $C$ is closed under additive inverses then $X$ is an undirected graph. This paper will only consider undirected connected circulant graphs.

In common with all Cayley graphs, circulant graphs are vertex transitive. With appropriate vertex labelling, the adjacency matrix of a circulant graph is a circulant matrix. By definition such a graph is regular, with the degree $d$ of each vertex equal to the order of $C$. If $n$ is odd then $\Z _n \setminus 0 $ has no elements of order $2$. Therefore $C$ has even order, say $d=2f$, and comprises $f$ complementary pairs of elements with one of each pair strictly between $0$ and $n/2$. The set of $f$ elements of $C$ between 0 and $n/2$ is defined to be the generator set $G$ for $X$. If $n$ is even then $\Z_n \setminus 0$ has just one element of order $2$, namely $n/2$. In this case $C$ is comprised of $f$ complementary pairs of elements, as for odd $n$, with or without the addition of the self-inverse element $n/2$. If $C$ has odd order, so that $d=2f+1$, then the value of the self-inverse element $n/2$ is fixed by $n$. Therefore for a circulant graph of given order and degree, its connection set $C$ is completely defined by specifying its generator set $G$. The order of the connection set is equal to the degree $d$ of the graph, and the order of the generator set $f$ is defined to be the dimension of the graph.
In summary, undirected circulant graphs of odd degree $d$ must have even order. They have dimension $f=(d-1)/2$. Graphs of even degree $d$ may have odd or even order. They have dimension $f=d/2$.

In the literature the symbol $d$ is variously used to define the degree or the diameter or the dimension of the graph. Adopting the terminology of Macbeth, Siagiova and Siran \cite {Macbeth}, we will use $d$ for degree and $k$ for diameter, and also $CC(d,k)$ for the order of an extremal circulant graph. In addition we use $f$ for the dimension of the graph.


\section{Upper and lower bounds}

The most widely known upper bounds for graph order for the degree-diameter problem are the Moore bounds for arbitrary graphs of given degree and diameter. However circulant graphs of degree 3 or more have girth of only 3 or 4 and so are significantly smaller than their Moore bound. Therefore this is not a useful upper bound in practice for extremal circulant graphs. Much more useful for circulant graphs is the upper bound for the order of the Cayley graph of any Abelian group given by Dougherty and Faber \cite{Dougherty}. 
For any dimension $f$ we consider $\Z^f$ with the canonical generators $\textbf{e}_i, 1\leq i\leq f$. For any Abelian group $G$ generated by $g_1, ..., g_f$ there is a unique homomorphism from $\Z^f$ onto $G$ which sends $\textbf{e}_i$ to $ g_i$ for all $i$. If $N$ is its kernel, then $G$ is isomorphic to $\Z^f/N$, and the Cayley graph of $G$ with the given generators is isomorphic to the Cayley graph of $\Z^f/N$ with the canonical generators for $\Z^f$. For any given diameter $k$, $S_{f,k}$ is defined to be the set of elements of $\Z^f$ which can be expressed as a word of length at most $k$ in the generators $\textbf{e}_i$ of $\Z^f$, taken positive or negative. Equivalently, $S_{f,k}$ is the set of points in $\Z^f$ distant at most $k$ from the origin under the $L^1$ (Manhattan) metric: $S_{f,k}=\{(x_1,...,x_f) \in \Z^f : \vert x_1 \vert +...+\vert x_f \vert \leq k\}$. Within the literature on coding theory and tiling problems it is usually called the $f$-dimensional Lee sphere of radius $k$. This leads to the following theorem.

\begin{theorem}
(Dougherty and Faber). Let $G,N$ and $g_1,...,g_f$ be as above. Then the undirected Cayley graph for an Abelian group $G$ and $g_1,...,g_f$ has diameter at most $k$ if and only if $S_{f,k}+N=\Z^f$.
\label {theoremD}
\end{theorem}

For any given dimension $f$, let $S(f,k)=\vert S_{f,k} \vert $. Then $S(f,k)$ gives an upper bound for the order of an Abelian Cayley graph of even degree $d=2f$ and diameter $k$, and therefore in particular of a circulant graph. In order for a graph to achieve this upper bound it is necessary that different combinations of multiples of the generators create paths of length up to $k$ from an arbitrary root vertex that lead to different vertices. This is equivalent to achieving an exact tiling of $\Z^f$ with Lee spheres $S_{f,k}$. Golomb and Welch \cite{Golomb} proved such a tiling is possible for 1 and 2 dimensions for any radius and for any dimension for radius 1. They conjectured that this is not possible for any graph of dimension $f\geq 3$ and radius $k\geq 2$. This conjecture is still open, although various authors have presented proofs of non-existence for 3, 4 and 5 dimensions. The paper by Horak \cite {Horak} covers all three of these dimensions.

From the definition it is easily shown that $S(f,k)$  satisfies the `square' recurrence relation $S(f,k)=S(f,k-1)+S(f-1,k)+S(f-1,k-1)$ which, along with boundary values $S(1,k)=1+2k$ for $k \geq 1$ and $S(f,1) = 1+2f$, for $f \geq 1$ enables the value to be calculated for any $f$ and $k$. Stanton and Cowan \cite{Stanton} derived an explicit formula for this relation:
\[S(f,k)= \sum _{i=0}^f 2^i{f\choose i} {k\choose i}.\]
We also have an asymptotic form: $S(f,k)=(2^f/f!)k^f+O(k^{f-1})$.

For a circulant graph of odd degree $d=2f+1$, having a single self-inverse generator, we note, following Dougherty and Faber \cite{Dougherty}, that the set of elements which can be written as a word of length at most $k$ in the generators of the group $\Z^f \times \Z_2$ is $(S_{f,k} \times 0) \cup (S_{k-1} \times 1)$, giving an upper bound for the order of the graph of $S(f,k)+S(f,k-1)$. For an Abelian Cayley graph of given degree a higher number of self-inverse generators would reduce the number of non self-inverse generators, thereby reducing the order of the polynomial in $k$. Thus this circulant graph upper bound is also the upper bound for any Abelian Cayley graph of the same degree and diameter.

We therefore define the upper bound $M_{AC}(d,k)$ for the order of an Abelian Cayley graph of degree $d$ and diameter $k$ as follows: 
\[M_{AC}(d,k) =
\begin{cases}
\ S(f,k) &\mbox{ for even } d, \mbox{ where } f=d/2\\
\ S(f,k)+S(f,k-1) &\mbox { for odd } d, \mbox{ where } f=(d-1)/2.
\end{cases}
\]
Thus we have the asymptotic form:
\[M_{AC}(d,k) =
\begin{cases}
\ (2^f / f!) k^f+O(k^{f-1}) &\mbox{ for even } d, \mbox{ where } f=d/2\\
\ (2^{f+1} / f!) k^f+O(k^{f-1}) &\mbox{ for odd } d, \mbox{ where } f=(d-1)/2.\\
\end{cases}
\]
Table \ref{table:3C} gives formulae for $M_{AC}(d,k)$ in terms of $k$ for $d \leq 9$.

\begin{table}[!htbp]
\caption{Formulae for upper bounds $M_{AC}(d,k)$ in terms of diameter $k$ for degree $d \leq 9$} 
\centering 
\begin{tabular} {l l}
\noalign {\vskip 2mm} 
\hline\hline 
\noalign {\vskip 1mm} 
Degree,$d$ & Upper bound, $M_{AC}(d,k)$ \\ 
\hline 
\noalign {\vskip 1mm} 
2 & $2k+1$ \\ 
3 & $4k$ \\
4 & $2k^2+2k+1$ \\
5 & $4k^2+2$ \\
6 & $(4k^3+6k^2+8k+3)/3$ \\
7 & $(8k^3+16k)/3$ \\ 
8 & $(2k^4+4k^3+10k^2+8k+3)/3$ \\
9 & $(4k^4+20k^2+6)/3$ \\
\hline 
\end{tabular}
\label{table:3C} 
\end{table}

As the Golomb-Welch conjecture has been confirmed for dimensions 3, 4 and 5, this means that no circulant graph of degree 6, 8 or 10 can achieve the upper bound $M_{AC}(d,k)$. Furthermore, as the upper bound for odd degree $d=2f+1$ also depends on $S(f,k)$, nor can any circulant graph of degree 7, 9 or 11 achieve its upper bound. The Golomb-Welch conjecture implies this also to be true for all higher dimensions.

Lower bounds on the order of extremal circulant graphs of even degree $d \geq 6$ were established by Chen and Jia in 1993 \cite{Chen}. For any dimension $f=d/2$ and diameter $k$ such that $k \geq f \geq 3$ let $a=\lfloor (k-f+3)/f \rfloor$. Then the lower bound of \cite{Chen} is given by
\[
CJ(d,k) = 2a \sum_{i=0}^{f-1} (4a)^i =  \frac{1}{2} \left (\frac{4}{f} \right )^fk^f+O(k^{f-1}).
\] 
 A graph of this order is constructed from the generator set $ \{ 1, 4a, (4a)^2,...,(4a)^{f-1} \}$.

For degrees $d=6$ and $d=8$ we have the following expressions.
\[CJ(6,k) = 32k^3 /27 + O(k^2) 
\]
\[CJ(8,k) = k^4 /2 + O(k^3)
\]


\section {Extremal and largest known circulant graphs up to dimension 3}

The upper bounds for Abelian Cayley graphs, $M_{AC}(d,k)$, are achieved for degree 2, 3 and 4 by circulant graphs. For degree 2, taking $\Z_{2k+1}$ and generator 1 (so that the connection set $C=\{ \pm 1\}$), the resultant graph is the cycle graph on $2k+1$ vertices which has diameter $k$, so that $CC(2,k) = 2k+1$.
For degree 3, taking $\Z _{4k}$ and generator 1, connection set $C= \{ \pm1,2k \}$, the graph is a cycle graph on $4k$ vertices with $2k$ edges added to join opposite pairs of vertices. Starting from vertex 0 and taking a path defined by edges +1 we can reach vertices $1,2,...,k$ with paths of length $\leq k$. By first taking edge $2k$ followed by edges +1 we reach vertices $2k,2k+1,...,3k-1$ with length $\leq k$. By taking these paths with edges $-1$ instead of $+1$ we can reach all the other vertices with a path of length $\leq k$. Hence the diameter is less than or equal to $k$. However we have $M_{AC}(3,k)=4k$. Therefore the specified graph is extremal and $CC(3,k)=4k$. 

For degree 4, Chen and Jia \cite{Chen} proved that $\Z_{2k^2+2k+1}$ with generator set $\{ 1,2k+1 \}$ has diameter $k$ for all $k$. As $M_{AC}(4,k)=2k^2+2k+1$, this proves the graph is extremal and $CC(4,k)=2k^2+2k+1$. For degree 5, Dougherty and Faber \cite{Dougherty} proved that the extremal solution for $k>1$ is $\Z_{4k^2}$ with generator set $\{ 1,2k-1 \}$ (connection set $\{\pm 1, \pm (2k-1), 2k^2 \}$) and order that is 2 less than $M_{AC}(5,k)$, giving $CC(5,k)=4k^2$. For degree 2, 3, 4 and 5 these extremal circulant graphs are also the largest Abelian Cayley graphs.

After degree 5 the situation becomes more difficult. Regarding graphs of three dimensions, Dougherty and Faber \cite{Dougherty} discovered families of circulant graphs of degree 6 and 7 which were proved by computer search to be extremal Abelian Cayley graphs for diameter $k \leq 18$ for degree 6, and for diameter $k \leq 10$ for degree 7. For both degree 6 and 7, the formula for the order of the solution, $DF(d,k)$, depends on the value of $k \pmod 3$. Tables \ref{table:3D} and \ref{table:3E} present these solutions alongside the corresponding expressions for the lower and upper bounds, $CJ(d,k)$ and $M_{AC}(d,k)$.

\begin{table} [!htbp]
\small
\caption{\small order of largest known solutions of degree 6, $DF(6,k)$, for arbitrary diameter $k\geq 2$, compared with lower bound $CJ(6,k)$ and upper bound $M_{AC}(6,k)$.} 
\centering 
\begin{tabularx} {\linewidth}{L L c l}
\noalign {\vskip 2mm}  
\hline\hline 
\noalign {\vskip 1mm} 
Diameter, $k$ & order, $DF(6,k)$ & & Lower bound, $CJ(6,k)$ \\
\hline  
\noalign {\vskip 1mm} 
$k \equiv 0 \pmod 3$ & $(32k^3+48k^2+54k+27)/27$ & & $(32k^3+24k^2+18k)/27$ \\ 
$k \equiv 1 \pmod 3$ & $(32k^3+48k^2+78k+31)/27$ & & $(32k^3-72k^2+66k-26)/27$ \\
$k \equiv 2 \pmod 3$ & $(32k^3+48k^2+54k+11)/27$ & & $(32k^3-168k^2+306k-196)/27$ \\  [1mm]
\hline 
\noalign {\vskip 1mm} 
Upper bound, $M_{AC}(6,k)$ & $(4k^3+6k^2+8k+3)/3$ & $=$ &$(36k^3+54k^2+72k+27)/27$ \\ 
\hline
\end{tabularx}
\label{table:3D} 
\end{table}

\begin{table} [!htbp]
\small
\caption{\small order of largest known solutions for degree 7, $DF(7,k)$, for arbitrary diameter $k\geq 3$, compared with upper bound $M_{AC}(7,k)$.} 
\centering 
\begin{tabularx} {\linewidth} {L r c l} 
\noalign {\vskip 2mm} 
\hline\hline 
\noalign {\vskip 1mm} 
Diameter, $k$ & & & order, $DF(7,k)$ \\ 
\hline 
\noalign {\vskip 1mm} 
$k \equiv 0 \pmod 3$ & & & $ (64k^3+108k)/27$ \\
$k \equiv 1 \pmod 3$ & & & $ (64k^3+60k-16)/27$ \\
$k \equiv 2 \pmod 3$ & & & $ (64k^3+60k+16)/27$ \\
\hline
\noalign {\vskip 1mm} 
Upper bound, $M_{AC}(7,k)$ & $(8k^3+16k)/3$ & $=$ &$(72k^3+144k)/27$ \\ 
\hline
\multicolumn {3} {l} {\footnotesize Note: $CJ(d,k)$ is only defined for even $d$.} 
\end{tabularx}
\label{table:3E} 
\end{table}

Dougherty and Faber \cite{Dougherty} proved the existence of the degree 6 graphs of order $DF(6,k)$ for all greater values of $k$, and they remain the largest Abelian Cayley graphs of three dimensions so far discovered. For the degree 6 graphs there is a unique solution up to isomorphism for diameter $k\equiv 1$ (mod 3), and for degree 7 there is a unique solution for $k\equiv 0$ (mod 3). For other values of $k$ there are two distinct isomorphism classes of graphs for both degree 6 and 7, where $k\ge 3$ for degree 7 as $DF(7,2)$ is not extremal. Generator sets for these solutions are shown in Table \ref{table:3F}.

\begin{table} [h]

\small
\caption{\small Generator sets for degree 6 and 7 graphs of order $DF(6,k)$ and $DF(7,k)$, for diameter $k$.} 
\centering 
\begin{tabularx} {\linewidth} { @ { } l l  l l l } 
\noalign {\vskip 1mm} 
\hline\hline 
\noalign {\vskip 1mm} 
Degree 6 &  \multicolumn {3} {l} {Generator set for isomorphism class 1} \\ 
\hline 
\noalign {\vskip 1mm} 
$k \equiv 0 \pmod 3  $ & & 1 & $(4k+3)/3$ & $(16k^2+12k+9)/9$  \\
$k \equiv 1 \pmod 3  $ & \multicolumn {2} {l} {none}   \\
$k \equiv 2 \pmod 3  $ & & 1 & $(4k+1)/3$ & $(16k^2+20k+13)/9$  \\
\hline \hline
\noalign {\vskip 1mm} 
Degree 6 & \multicolumn {3} {l} {Generator set for isomorphism class 2} \\ 
\hline 
\noalign {\vskip 1mm} 
$k \equiv 0 \pmod 3  $ & & 1 & $(8k^2+6k)/9$ & $(8k^2+18k+18)/9$  \\
$k \equiv 1 \pmod 3  $ & & 1 & $(8k^2+2k+8)/9$ & $(8k^2+14k+14)/9$  \\
$k \equiv 2 \pmod 3  $ & & 1 & $(8k^2-2k+8)/9$ & $(8k^2+10k+2)/9$  \\
\hline \hline
\noalign {\vskip 1mm} 
Degree 7 &  \multicolumn {3} {l} {Generator set for isomorphism class 1} \\ 
\hline 
\noalign {\vskip 1mm} 
$k \equiv 0 \pmod 3  $ & \multicolumn {2} {l} {none}  \\
$k \equiv 1 \pmod 3  $ & & 1 & $(4k-1)/3$ & $(16k^2+4k+7)/9$  \\
$k \equiv 2 \pmod 3  $ & & 1 & $(4k+1)/3$ & $(16k^2-4k+7)/9$  \\
\hline \hline
\noalign {\vskip 1mm} 
Degree 7 & \multicolumn {3} {l} {Generator set for isomorphism class 2} \\ 
\hline 
\noalign {\vskip 1mm} 
$k \equiv 0 \pmod 3  $ & & 1 & $(32k^3-24k^2+36k-27)/27$ & $(32k^3-24k^2+72k-27)/27$  \\
$k \equiv 1 \pmod 3  $ & & 1 & $(32k^3-24k^2+24k-5)/27$ & $(32k^3-24k^2+60k-41)/27$  \\
$k \equiv 2 \pmod 3  $ & & 1 & $(32k^3-24k^2-25)/27$ & $(32k^3-24k^2+36k+11)/27$  \\ 
\hline
\end{tabularx}
\label{table:3F} 
\end{table}


\section {Largest known circulant graphs of degree 8 and 9}

Graphs of dimension 4 have degree 8 or 9. As with Dougherty and Faber's approach for dimension 3, an exhaustive computer search was conducted for potential solutions using all feasible generator sets within relevant ranges. For small diameter this process worked well and enabled the discovery of a family of graphs of degree 8 which are larger than the lower bound $CJ(8,k)$ for any diameter $k$, and similarly for degree 9. However the order of graphs on generator sets of dimension 4 increases with diameter much more quickly than for dimension 3, as well as the number of possible permutations for each order. This means that the calculations to prove the extremality of a candidate graph by continuing the search up to the relevant upper bound, $M_{AC}(d,k)$, quickly exceed the available computing power. Therefore the discovered candidate families of dimension 4 graphs have only been proven to be extremal for a rather limited range of diameters, $k \leq 7$ for degree 8 and $k \leq 6$ for degree 9. The results for degree 8 are shown in Table \ref {table:5A}. 

\begin {table} [!htbp]
\small
\caption{\small Largest known circulant graphs of degree 8.} 
\centering 
\begin{tabular}{ @ { } c  r c l r r l}
\noalign {\vskip 2mm}  
\hline\hline 
\noalign {\vskip 1mm} 
Diameter & order & Distinct & Generator set & Upper bound & Limit of & Status \\ 
$k$ & $L(8,k)$ & solutions & & $M_{AC}(8,k)$ & search & \\
\hline 
\noalign {\vskip 1mm} 
2 & 35 & 2 & $1,6,7,10$ & 41 & 41 & Extremal \\
 & & & $1,7,11,16$ & \\
3 & 104 & 1 & $1,16,20,27$ & 129 & 129 & Extremal \\
4 & 248 & 1 & $1,61,72,76$ & 321 & 321 & Extremal \\
5 & 528 & 1 & $1,89,156,162$ & 681 & 681 & Extremal \\
6 & 984 & 1 & $1,163,348,354$ & 1289 & 1289 & Extremal \\
7 & 1712 & 1 & $1,215,608,616$ & 2241 & 2241 & Extremal  \\
8 & 2768 & 1 & $1,345,1072,1080$ & 3649 &-  & Largest known  \\
9 & 4280 & 1 & $1,429,1660,1670$ & 5641 & - & Largest known  \\
10 & 6320 & 1 & $1,631,2580,2590$ & 8361 &-  & Largest known  \\
11 & 9048 & 1 & $1,755,3696,3708$ & 11969 &-  & Largest known  \\
12 & 12552 & 1 & $1,1045,5304,5316$ & 16641 &-  & Largest known  \\
13 & 17024 & 1 & $1,1217,7196,7210$ & 22569 & - & Largest known  \\
14 & 22568 & 1 & $1,1611,9772,9786$ & 29961 & - & Largest known  \\
15 & 29408 & 1 & $1,1839,12736,12752$ & 39041 &-  & Largest known  \\
16 & 37664 & 1 & $1,2353,16608,16624$ & 50049 & - & Largest known  \\
\hline
\end{tabular}
\label{table:5A}
\end{table}

For degree 8 the following quartic polynomials in $k$ determine the order of these solutions for diameter $k \geq 3$:
\[L(8,k) =
\begin{cases}
\ (k^4+2k^3+6k^2+4k)/2 &\mbox{ for } k \equiv 0 \pmod 2 \\
\ (k^4+2k^3+6k^2+6k+1)/2 &\mbox{ for } k \equiv 1 \pmod 2 \\
\end{cases}
\]
Over the range of diameters checked there is just one unique graph up to isomorphism for each $k \geq 3$. The leading coefficient of 1/2 equals the lower bound value in the formula for $CJ(8,k)$ and is below the upper bound value of 2/3 in $M_{AC}(8,k)$. See Table \ref {table:5B}.

\begin {table} [!htbp]
\small
\caption {\small order, $L(8,k)$, and generator sets of largest known circulant graphs of degree 8 for diameter $k \geq 3$.}
\centering
\begin {tabular} {l l l}
\noalign {\vskip 2mm} 
\hline \hline
\noalign {\vskip 1mm} 
& $k \equiv 0 \pmod 2$ & $k \equiv 1 \pmod 2$ \\
\hline
\noalign {\vskip 2mm} 
order, $L(8,k)$ & $(k^4+2k^3+6k^2+4k)/2$ & $(k^4+2k^3+6k^2+6k+1)/2$ \\
\noalign {\vskip 1mm}
\hline
\noalign {\vskip 1mm} 
Generator & 1 & 1 \\
set & $(k^3+2k^2+6k+2)/2$ & $(k^3+k^2+5k+3)/2$ \\
 & $(k^4+4k^2-8k)/4$ & $(k^4+2k^2-8k-11)/4$ \\
 & $(k^4+4k^2-4k)/4$ & $(k^4+2k^2-4k-7)/4$ \\
\hline
\noalign {\vskip 1mm} 
Lower bound & & \\
$CJ(8,k)$ & $(k^4-15k^3+85k^2-215k+204)/2$ & $(k^4-3k^3+4k^2-2k)/2$ \\
 & for $k \equiv 0 \pmod 4$ &  for $k \equiv 1 \pmod 4$ \\
& $(k^4-7k^3+19k^2-23k+10)/2$ & $k^4-11k^3+46k^2-86k+60)/2$ \\
& for $k \equiv 2 \pmod 4$ & for $k \equiv 3 \pmod 4$ \\
Upper bound & & \\
$M_{AC}(8,k)$ & $(2k^4+4k^3+10k^2+8k+3)/3$ & $(2k^4+4k^3+10k^2+8k+3)/3$ \\
\hline
\end {tabular}
\label {table:5B}
\end {table}

For $k=2$ the formula gives a graph of order 32 whereas the optimal order is 35 with two non-isomorphic solutions. For $3 \leq k \leq 7$ the resulting graphs have been proven extremal by exhaustive computer search up to the upper bound $M_{AC}(8,k)$. The existence of these graphs for all $k$ is proved in the next section. They are the best solutions so far discovered for any $k \geq 3$ and are conjectured to be extremal.

The results for degree 9 are shown in Table \ref {table:5E}.

\begin {table} [!htbp]
\small
\caption{\small Largest known circulant graphs of degree 9.} 
\centering 
\begin{tabular}{ @ { } c  r c l r r l} 
\noalign {\vskip 2mm} 
\hline\hline 
\noalign {\vskip 1mm} 
Diameter & order & Distinct & Generator set* & Upper bound & Limit of & Status \\ 
$k$ & $L(9,k)$ & solutions & & $M_{AC}(9,k)$ & search & \\
\hline 
\noalign {\vskip 1mm} 
2 & 42 & 2 & $1,5,14,17$ & 50 & 50 & Extremal \\
 & & & $2,7,8,10$ & \\
3 & 130 & 4 & $1,8,14,47$ & 170 & 170 & Extremal \\
 & & & $1,8,20,35$ & \\
 & & & $1,26,49,61$ & \\
 & & & $2,8,13,32$ & \\
4 & 320 & 1 & $1,15,25,83$ & 450 & 450 & Extremal \\
5 & 700 & 2 & $1,5,197,223$ & 1002 & 1002 & Extremal \\
 & & & $1,45,225,231$ & \\
6 & 1416 & 1 & $1,7,575,611$ & 1970 & 1970 & Extremal \\
7 & 2548 & 2 & $1,7,521,571$ & 3530 & - & Largest known \\
 & & & $1,581,1021,1029$ & \\
8 & 4304 & 1 & $1,9,1855,1919$ & 5890 & - & Largest known  \\
9 & 6804 & 2 & $1,9,1849,1931$ & 9290 & - & Largest known  \\
 & & & $1,1305,1855,1863$ & \\
10 & 10320 & 1 & $1,11,4599,4699$ & 14002 & - & Largest known  \\
11 & 15004 & 2 & $1,11,3349,3471$ & 20330 & - & Largest known  \\
 & & & $1,4851,6655,6667$ & \\
12 & 21192 & 1 & $1,13,9647,9791$ & 28610 & - & Largest known  \\
13 & 29068 & 2 & $1,13,7741,7911$ & 39210 &  & Largest known  \\
 & & & $1,5083,7929,7943$ & \\
14 & 39032 & 1 & $1,15,18031,18227$ & 52530 & - & Largest known  \\
15 & 51300 & 2 & $1,15,11857,12083$ & 69002 & - & Largest known  \\
 & & & $1,5835,15075,15089$ & \\
16 & 66336 & 1 & $1,17,30975,31231$ & 89090 & - & Largest known  \\
\hline
\multicolumn {7} {l} {\footnotesize * for each isomorphism class of graphs just one of the generator sets is listed} 
\end{tabular}
\label{table:5E}
\end{table}

For degree 9 the following quartic polynomials in $k$ determine the order of the largest known solutions for diameter $k \geq 5$:
\[ L(9,k) = 
\begin{cases}
\ k^4+3k^2+2k &\mbox{ for } k \equiv 0 \pmod 2 \\
\ k^4+3k^2 &\mbox{ for } k \equiv 1 \pmod 2 \\
\end{cases}
\]
This may be compared with the upper bound $M_{AC}(9,k) = (4k^4+20k^2+6)/3$.

Over the range of diameters $k\geq5$ checked there is a unique solution up to isomorphism for each even diameter and two for each odd diameter. For isomorphism class 1, for each diameter there are three or four generator sets which include the element $1$. Formulae for the generator sets for isomorphism class 1 are listed in Table \ref {table:5F}.

\begin {table} [h]
\footnotesize
\caption {\small order, $L(9,k)$, of largest known degree 9 circulant graphs for diameter $k \geq 5$, and generator sets for isomorphism class 1.} 
\centering
\begin{tabularx} {\linewidth} { l l l l l }
\noalign {\vskip 2mm} 
\hline \hline
\noalign {\vskip 1mm} 
& & $k \equiv 0 \pmod 2$ & $k \equiv 1 \pmod 4$ & $k \equiv 3 \pmod 4$ \\
\hline
\noalign {\vskip 2mm} 
order, $L(9,k)$ & & $k^4+3k^2+2k$ & $k^4+3k^2$ & $k^4+3k^2$ \\
\hline
\noalign {\vskip 1mm}
Generator & & 1 & 1 & 1 \\
set 1 & & $k+1$ & $k$ & $k$ \\
& & $(k^4-k^3+2k^2-2)/2$ & $(k^4+k^3+k^2+3k-2)/4$ & $(k^4-k^3+k^2-3k-2)/4$  \\
& & $(k^4-k^3+4k^2-2)/2$ & $(k^4+k^3+5k^2+3k+2)/4$ & $(k^4-k^3+5k^2-3k+2)/4$ \\
\hline \hline
\noalign {\vskip 1mm} 
& & $k \equiv 0 \pmod 2$ & $k \equiv 1 \pmod 2$ & \\
\hline
\noalign {\vskip 1mm} 
Generator & & 1 & 1 & \\
set 2 & & $k^3-k^2+3k-1$ & $k^3+2k$ & \\
& & $k^3-k^2+4k-1$ & $k^3+3k+1$ & \\
& & $3k^3-2k^2+10k-1$ & $k^3+k^2+3k+2$ & \\
\hline
\noalign {\vskip 1mm} 
Generator & & 1 & 1 & \\
set 3 & & $(k^3+2k+2)/2$ & $k^3-k^2+3k-2$ & \\
& & $(k^4-k^3+4k^2-2k+2)/2$ & $k^3+2k$ & \\
& & $(k^4+2k^2+2k-2)/2$ & $k^3+3k-1$ & \\
\hline \hline
\noalign {\vskip 1mm} 
& & $k \equiv 0 \pmod 6$ & $k \equiv 2 \pmod 6$ & \\
\hline
\noalign {\vskip 1mm} 
Generator & & 1 & 1 & \\
set 4* & & $(k^4+k^3+k^2+6k-3)/3$ & $(k^4-k^3+2k^2-2k-3)/3$ & \\
& & $(k^4+k^3+4k^2+3k+3)/3$ & $(k^4-k^3+2k^2+k-3)/3$ & \\
& & $(k^4+k^3+4k^2+6k+3)/3$ & $(k^4-k^3+5k^2-2k+3)/3$ & \\
\hline
\multicolumn {4} {l} {* No solutions for $k \equiv 4 \pmod 6$ or $k \equiv 1 \pmod 2$} 
\end {tabularx}
\label {table:5F}
\end {table}

Factors establishing isomorphisms between the graphs generated by the various generating sets of isomorphism class 1 are listed in Table \ref {table:5G}.

\begin{table} [!htbp]
\small
\caption{\small Factors transforming generator set 1 into the other generator sets for degree 9 graphs of isomorphism class 1, for any diameter $k$.} 
\centering 
\begin{tabularx} {\linewidth}{L l l }
\noalign {\vskip 2mm}  
\hline\hline 
\noalign {\vskip 1mm} 
Diameter, $k$ & For generator set 2 & For generator set 3 \\
\hline  
\noalign {\vskip 1mm} 
$k \equiv 0 \pmod 2$ & $g_4=3k^3-2k^2+10k-1$  & $g_3=(k^4-k^3+4k^2-2k+2)/2$ \\ 
\noalign {\vskip 1mm} 
$k \equiv 1 \pmod 2$ & $g_4=k^3+k^2+3k+2$  & $g_2=k^3-k^2+3k-2$ \\
\noalign {\vskip 1mm} 
\hline\hline 
\noalign {\vskip 1mm} 
Diameter, $k$ & For generator set 4 & \\
\hline  
\noalign {\vskip 1mm} 
$k \equiv 0 \pmod 6$ & $g_2=(k^4+k^3+k^2+6k-3)/3$ &  \\  [1mm]
$k \equiv 2 \pmod 6$ & $g_4=(k^4-k^3+5k^2-2k+3)/3$ &  \\  [1mm]
\hline 

\end{tabularx}
\label{table:5G} 
\end{table}

For isomorphism class 2 the formulae for the generator sets were more difficult to discover. It emerges that the solution depends on the diameter $k \pmod {14}$. This is a most surprising result for a system with dimension 4 and degree 9, neither value having a common factor with 7. For each odd diameter there are one or two generator sets which include the element 1. One sequence omits the case $k\equiv 9 \pmod{14}$ with each solution having a pair of generators differing by $k+1$. The other omits the case $k\equiv 5 \pmod{14}$ with each having a pair differing by $k-1$. The formulae for these generator sets are listed in Table \ref {table:5J}.

\begin {table} [h]
\footnotesize
\caption {\small Generator sets for isomorphism class 2 of degree 9 circulant graphs of order $L(9,k)$ for diameter $k \geq 5$.} 
\centering
\begin{tabularx} {\linewidth} { l l l l l }
\noalign {\vskip 2mm} 
\hline \hline
\noalign {\vskip 1mm} 
Diameter, $k$ & & Generator set 1 && Generator set 2  \\
\hline
\noalign {\vskip 1mm}
$k\equiv 1 \pmod{14}$ & & 1 & & 1  \\
& & $(k^4+k^3+5k^2)/7$ & & $(k^4-3k^3+2k^2-7k)/7$  \\
& & $(k^4+k^3+5k^2+7k+7)/7$ && $(2k^4+k^3+4k^2)/7$   \\
& & $(3k^4+3k^3+8k^2+7k)/7$ && $(2k^4+k^3+4k^2+7k-7)/7$  \\
\noalign {\vskip 2mm} 
$k\equiv 3 \pmod{14}$ & & 1 && 1  \\
& & $(k^4-k^3+k^2-7k-7)/7$ & & $(2k^4-3k^3+5k^2-7k)/7$  \\
& & $(k^4-k^3+k^2)/7$ && $(3k^4-k^3+11k^2-7k+7)/7$   \\
& & $(3k^4-3k^3+10k^2-7k)/7$ && $(3k^4-k^3+11k^2)/7$  \\
\noalign {\vskip 2mm} 
$k\equiv 5 \pmod{14}$ & & 1 && *  \\
& & $(k^4-3k^3+4k^2-7)/7$ & &  \\
& & $(2k^4+k^3+8k^2)/7$ &&    \\
& & $(2k^4+k^3+8k^2+7k+7)/7$ &&  \\
\noalign {\vskip 2mm} 
$k\equiv 7 \pmod{14}$ & & 1 && 1  \\
& & $(2k^4-3k^3+7k^2-7k)/7$ & & $(2k^4+3k^3+7k^2+7k)/7$  \\
& & $(3k^4-k^3+7k^2-7k-7)/7$ && $(3k^4+k^3+7k^2)/7$   \\
& & $(3k^4-k^3+7k^2)/7$ && $(3k^4+k^3+7k^2+7k-7)/7$  \\
\noalign {\vskip 2mm} 
$k\equiv 9 \pmod{14}$ & & *  && 1  \\
& &                                        & & $(k^4+3k^3+4k^2+7k)/7$  \\
& &                                         && $(2k^4-k^3+8k^2-7k+7)/7$   \\
& &                                        && $(2k^4-k^3+8k^2)/7$  \\
\noalign {\vskip 2mm} 
$k\equiv 11 \pmod{14}$ & & 1 && 1  \\
& & $(2k^4+3k^3+5k^2+7k)/7$ & & $(k^4+k^3+k^2)/7$  \\
& & $(3k^4+k^3+11k^2)/7$ && $(k^4+k^3+k^2+7k-7)/7$   \\
& & $(3k^4+k^3+11k^2+7k+7)/7$ && $(3k^4+3k^3+10k^2+7k)/7$  \\
\noalign {\vskip 2mm} 
$k\equiv 13 \pmod{14}$ & & 1 && 1  \\
& & $(k^4+3k^3+2k^2+7)/7$ & & $(k^4-k^3+5k^2-7k+7)/7$  \\
& & $(2k^4-k^3+4k^2-7k-7)/7$ && $(k^4-k^3+5k^2)/7$   \\
& & $(2k^4-k^3+4k^2)/7$ && $(3k^4-3k^3+8k^2-7k)/7$  \\
\noalign {\vskip 1mm} 
\hline
\multicolumn {5} {l} {* No solutions for $k \equiv 9 \pmod {14}$ for generator set 1 or for $k \equiv 5 \pmod {14}$ for generator set 2} 
\end {tabularx}
\label {table:5J}
\end {table}

Factors establishing isomorphisms between the graphs generated by the generating sets of isomorphism class 2 are listed in Table \ref {table:5K}.

\begin{table} [!htbp]
\small
\caption{\small Factors transforming generator set 1 into generator set 2 for degree 9 graphs of isomorphism class 2, for any diameter $k$.} 
\centering
\begin{tabular} { l l l }
\noalign {\vskip 2mm} 
\hline \hline
\noalign {\vskip 1mm} 
Diameter, $k$ & & Factor (elements from generator set 2)   \\
\hline
\noalign {\vskip 1mm}
$k\equiv 1 \pmod{14}$ & &  $g_4=(2k^4+k^3+4k^2+7k-7)/7$   \\
$k\equiv 3 \pmod{14}$ & & $g_3=(3k^4-k^3+11k^2-7k+7)/7$   \\
$k\equiv 5 \pmod{14}$ & & *   \\
$k\equiv 7 \pmod{14}$ & & $g_4=(3k^4+k^3+7k^2+7k-7)/7$  \\
$k\equiv 9 \pmod{14}$ & & *   \\
$k\equiv 11 \pmod{14}$ & & $g_3=(k^4+k^3+k^2+7k-7)/7$   \\
$k\equiv 13 \pmod{14}$ & & $g_2=(k^4-k^3+5k^2-7k+7)/7$  \\
\noalign {\vskip 1mm} 
\hline
\multicolumn {3} {l} {* No solutions for $k \equiv 9 \pmod {14}$ for generator set 1 or for $k \equiv 5 \pmod {14}$ for generator set 2} 
\end {tabular}
\label{table:5K} 
\end{table}

In order to establish that the graphs of isomorphism class 2 are not isomorphic to class 1, at least for the diameters checked, spectral analysis is used. The spectrum of a circulant graph may be determined straightforwardly using the following presentation by Nguyen, \cite {Nguyen}, of the standard formula for the eigenvalues.

\begin{theorem}
(Nguyen) Let A be the adjacency matrix of a circulant graph on n vertices, where $c_i=c_{n-i}=1$ if vertices i and n-i are adjacent and $0$ otherwise, and let Sp(A) be its spectrum. If n is odd then
\[Sp(A)= \Bigg \{ \sum _{i=1}^{(n-1)/2} 2c_i \mbox{\rm cos} \frac{2li\pi}{n} : 1\leq l \leq n \Bigg \}.\]
If n is even then
\[Sp(A)= \Bigg \{ \sum _{i=1}^{(n-2)/2} 2c_i \mbox{\rm cos} \frac{2li\pi}{n}+c_{n/2} \mbox{\rm cos} l\pi : 1\leq l \leq n \Bigg \}.\]

\end{theorem}

Some differences in the spectral analysis between the two classes are listed in Table \ref {table:5H}, proving that the two isomorphism classes are distinct. The same approach can be used to confirm the two isomorphism classes for degree 6 and 7.

\begin {table} [!htbp]
\small
\caption{\small Spectral analysis of graphs of isomorphism classes 1 and 2 for diameter 5, 7, 9 and 11 .} 
\centering 
\begin{tabularx}{\linewidth}{ c c c c c c c c c c c } 
\noalign {\vskip 2mm} 
\hline\hline 
\noalign {\vskip 1mm}
Diameter & order & &  \multicolumn {2} {l} {Positive eigenvalues} & &  \multicolumn {2} {l} {Zero eigenvalues} & &  \multicolumn {2} {l} {Negative eigenvalues} \\ 
$k$ & $L(9,k)$ & & Class 1 & Class 2 & & Class 1 & Class 2 & & Class 1 & Class 2 \\ 
\hline 
\noalign {\vskip 1mm} 
5 & 700 & & 315 & 319 & & 0 & 0 & & 385 & 381 \\
7 & 2548 & & 1215 & 1211 & & 0 & 0 & & 1333 & 1337 \\
9 & 6804 & & 3343 & 3347 & & 2 & 0 & & 3459 & 3457 \\
11 & 15004 &&  7539 & 7529 & & 0 & 0 & & 7465 & 7475 \\
\noalign {\vskip 1mm}
\hline

\end{tabularx}
\label{table:5H}
\end{table}

For diameter $k \leq 4$ the graphs determined by the formulae are not optimal. For $k=5$ and $k=6$ the resulting graphs have been proven extremal by checking up to the upper bound $M_{AC}(9,k)$. The existence of these graphs has also been confirmed by computer for all diameters $k \leq 80$. They are the largest circulant graphs so far discovered for any diameter $k \geq 5$ and are conjectured to be extremal.


\section {Existence proof for the degree 8 graph of order $L(8,k)$ for all diameters}

In this section we prove the existence of the degree 8 circulant graph of order $L(8,k)$ for all diameters $k \geq 2$. This proof closely follows the approach taken by Dougherty and Faber in their proof of the existence of the degree 6 graph of order $DF(6,k)$ for all diameters $k \geq 2$ \cite{Dougherty}. 

\begin{theorem}
For all $k\geq 2$, there is an undirected Cayley graph on four generators of a cyclic group which has diameter k and order $L(8,k)$, where
\[ L(8,k)=
\begin{cases}
(k^4+2k^3+6k^2+4k)/2 &\mbox{ if } k\equiv 0  \pmod 2\\
(k^4+2k^3+6k^2+6k+1)/2 &\mbox{ if } k\equiv 1  \pmod 2 
\end{cases}
\]
Moreover for $k \equiv 0 \pmod 2$ a generator set is
$\{1, (k^3+2k^2+6k+2)/2, (k^4+4k^2-8k)/4, (k^4+4k^2-4k)/4\}$, \\
and for $k\equiv 1 \pmod 2$,
$\{1, (k^3+k^2+5k+3)/2, (k^4+2k^2-8k-11)/4, (k^4+2k^2-4k-7)/4\}$.
\label{theorem:A}
\end{theorem}
\begin{proof}
We will show the existence of four-dimensional lattices $L_k \subseteq \Z^4$ such that $\Z^4/L_k$ is cyclic, $S_{f,k}+L_k=\Z^4$, where $S_{f,k}$ is the set of points in $\Z^4$ at a distance of at most $k$ from the origin under the $l^1$ (Manhattan) metric, and $\vert \Z^4 : L_k\vert = L(8,k) $ as specified in the theorem. Then, by Theorem \ref {theoremD}, the resultant Cayley graph has diameter at most $k$. 

Let $a=\begin {cases} k/2 & \mbox { for } k \equiv 0 \pmod 2\\ (k+1)/2 & \mbox { for } k \equiv 1 \pmod 2. \end {cases}$

For $k \equiv 0 \pmod 2$ let $L_k$ be defined by four generating vectors as follows:
\[
\begin{array}{rcl}
\textbf{v}_1&=&(-a-1,a+1,a,-a+1)\\
\textbf{v}_2&=&(a-1,a+1,a+1,-a)\\
\textbf{v}_3&=&(-a-1,-a+1,a+1,-a)\\
\textbf{v}_4&=&(-a,-a,a,a+1)
\end{array}
\] 
Then the following vectors are in $L_k$:
\[
\begin{array}{l}
-(2a^2+2a+1)\textbf{v}_1 + (2a^2+a+2)\textbf{v}_2 - (a+2)\textbf{v}_3 +\textbf{v}_4 = (4a^3+4a^2+6a+1, -1, 0, 0),\\
-(2a^3-1)\textbf{v}_1 + (2a^3-a^2+2a-2)\textbf{v}_2 - (a^2+a-1)\textbf{v}_3 +(a-1)\textbf{v}_4 = (4a^4+4a^2-4a, 0, -1, 0),\\
-2a^3\textbf{v}_1 + (2a^3-a^2+2a-1)\textbf{v}_2 - (a^2+a-1)\textbf{v}_3 +(a-1)\textbf{v}_4 = (4a^4+4a^2-2a, 0, 0, -1)\\
\end{array}
\]
Hence we have $\textbf{e}_2 =  (4a^3+4a^2+6a+1)\textbf{e}_1,  \textbf{e}_3 =  (4a^4+4a^2-4a)\textbf{e}_1$  and $\textbf{e}_4 =  (4a^4+4a^2-2a)\textbf{e}_1$ in $\Z^4/L_k$, and so $\textbf{e}_1$ generates $\Z^4/L_k$.

Also det $\left ( \begin{array} {c} \textbf{v}_1 \\ \textbf{v}_2 \\ \textbf{v}_3 \\ \textbf{v}_4 \end {array} \right ) = $ det
$\left (
\begin{array} {l r r r}
8a^4+8a^3+12a^2+4a & 0 & 0 & 0 \\
4a^3+4a^2+6a+1 & -1 & 0 & 0 \\
4a^4+4a^2-4a & 0 & -1 & 0 \\
4a^4+4a^2-2a & 0 & 0 & -1 
\end {array} \right ) 
\newline
\newline
= -(8a^4+8a^3+12a^2+4a) = -(k^4+2k^3+6k^2+4k)/2= -L(8,k)$, as in the statement of the theorem. 

Thus $\Z^4/L_k$ is isomorphic to $\Z_{L(8,k)}$ via an isomorphism taking $\textbf{e}_1, \textbf{e}_2, \textbf{e}_3, \textbf{e}_4$ to $1$, $4a^3+4a^2+6a+1, 4a^4+4a^2-4a,  4a^4+4a^2-2a$. As $a=k/2$ this gives the first generator set specified in the theorem: $\{1, (k^3+2k^2+6k+2)/2, (k^4+4k^2-8k)/4, (k^4+4k^2-4k)/4\}$.

Similarly for $k \equiv 1 \pmod 2$ let $L_k$ be defined by four generating vectors as follows:
\[
\begin{array}{rcl}
\textbf{v}_1&=&(-a+1,a+1,-a+1,a)\\
\textbf{v}_2&=&(a+1,a+1,-a+2,a-1)\\
\textbf{v}_3&=&(-a-1,a-1,a-1,-a)\\
\textbf{v}_4&=&(-a,a,a,a-1)
\end{array}
\]
In this case the following vectors are in $L_k$:
\newline
$-(2a^2+a+2)\textbf{v}_1 + (2a^2+2a+1)\textbf{v}_2 - a\textbf{v}_3 -\textbf{v}_4 = (4a^3-4a^2+6a-1, -1, 0, 0)$,
\newline
$-(2a^3-a^2-2a-2)\textbf{v}_1 + (2a^3-4a-1)\textbf{v}_2 - (a^2-a-1)\textbf{v}_3 -(a-1)\textbf{v}_4 = (4a^4-8a^3+8a^2-8a, 0, -1, 0)$,
\newline
$-(2a^3-a^2-2a-1)\textbf{v}_1 + (2a^3-4a)\textbf{v}_2 - (a^2-a-1)\textbf{v}_3 -(a-1)\textbf{v}_4 = (4a^4-8a^3+8a^2-6a, 0, 0, -1)$.

Hence we have $\textbf{e}_2 = (4a^3+4a^2+6a-1)\textbf{e}_1$,  $\textbf{e}_3 = (4a^4-8a^3+8a^2-8a)\textbf{e}_1$ and $\textbf{e}_4=  (4a^4-8a^3+8a^2-6a)\textbf{e}_1$, in $\Z^4/L_k$, and so $\textbf{e}_1$ generates $\Z^4/L_k$.

Also det $\left ( \begin{array} {c} \textbf{v}_1 \\ \textbf{v}_2 \\ \textbf{v}_3 \\ \textbf{v}_4 \end {array} \right ) = $ det
$\left ( \begin{array} {l r r r}
8a^4-8a^3+12a^2-4a & 0 & 0 & 0 \\
4a^3-4a^2+6a-1 & -1 & 0 & 0 \\
4a^4-8a^3+8a^2-8a & 0 & -1 & 0 \\
4a^4-8a^3+8a^2-6a & 0 & 0 & -1 \end {array} \right ) 
\newline
\newline
= -(8a^4-8a^3+12a^2-4a) = -(k^4+2k^3+6k^2+6k+1)/2= -L(8,k)$, as in the statement of the theorem.

Thus $\Z^4/L_k$ is isomorphic to $\Z_{L(8,k)}$ with generators $1, 4a^3-4a^2+6a-1, 4a^4-8a^3+8a^2-8a, 4a^4-8a^3+8a^2-6a$.
As $a=(k+1)/2$ in this case, this gives the second generator set specified in the theorem: $\{1, (k^3+k^2+5k+3)/2, (k^4+2k^2-8k-11)/4, (k^4+2k^2-4k-7)/4\}$.

It remains to show that $S_{f,k}+L_k=\Z^4$. First we consider the case $k\equiv 0 \pmod 2$. For $k=2$, it is straightforward to show directly that $\Z_{32}$ with generators $1, 4, 6, 15$ has diameter 2. So we assume $k\geq4$, so that $a \geq 2$. Now let
\[
\begin{array}{rcl}
\textbf{v}_5&=\textbf{v}_1-\textbf{v}_3+\textbf{v}_4=&(-a,a,a-1,a+2)\\
\textbf{v}_6&=\textbf{v}_1-\textbf{v}_2-\textbf{v}_4=&(-a,a,-a-1,-a)\\
\textbf{v}_7&=\textbf{v}_1-\textbf{v}_2-\textbf{v}_3=&(-a+1,a-1,-a-2,a+1)\\
\textbf{v}_8&=\textbf{v}_2-\textbf{v}_3+\textbf{v}_4=&(a,a,a,a+1)
\end{array}
\]
with $\textbf{v}_1, \textbf{v}_2, \textbf{v}_3, \textbf{v}_4$ as defined for $k \equiv 0$ (mod 2). Then the 16 vectors $\pm\textbf{v}_i$ for $i=1,...,8$ provide one element of $L_k$ lying strictly within each of the 16 orthants of $\Z^4$. Most of the coordinates of these vectors have absolute value at most $a+1$. Only $\pm \textbf{v}_5$ and $\pm \textbf{v}_7$ each have one coordinate with absolute value equal to $a+2$.

Now we consider the case $k\equiv 1 \pmod 2$. For $k=3$ it may be shown directly that $\Z_{104} $ with generators $1, 16, 20, 27$ has diameter 3. So we assume $k \geq5$, so that $a \geq3$, and let 
\[
\begin{array}{rcl}
\textbf{v}_5&=\textbf{v}_1-\textbf{v}_2-\textbf{v}_4=&(-a,-a,-a-1,-a+2)\\
\textbf{v}_6&=\textbf{v}_2+\textbf{v}_3-\textbf{v}_4=&(a,a,-a+1,-a)\\
\textbf{v}_7&=\textbf{v}_1+\textbf{v}_3-\textbf{v}_4=&(-a,a,-a,-a+1)\\
\textbf{v}_8&=\textbf{v}_1-\textbf{v}_2-\textbf{v}_3=&(-a+1,-a+1,-a,a+1)
\end{array}
\]
with $\textbf{v}_1, \textbf{v}_2, \textbf{v}_3, \textbf{v}_4$ as defined for $k \equiv 1$ (mod 2), so that the 16 vectors $\pm \textbf{v}_i$ provide one element of $\textbf{L}_k$ lying strictly within each of the orthants of $\Z^4$. In this case all the coordinates of these vectors have absolute value at most $a+1$.

We must show that each $\textbf{x}\in \Z^4$ is in $S_{f,k}+L_k$, which means that for any $\textbf{x} \in \Z^4$ we need to find a $\textbf{w} \in L_k$ such that $\textbf{x}-\textbf{w} \in S_{f,k}$. However $\textbf{x}-\textbf{w} \in S_{f,k}$ if and only if $\delta (\textbf{x},\textbf{w})\leq k$, where $\delta$ is the $l^1$ metric on $\Z^4$.
If $\textbf{x}, \textbf{y}, \textbf{z} \in \Z^4$ and each coordinate of $\textbf{y}$ lies between the corresponding coordinate of $\textbf{x}$ and $\textbf{z}$ or is equal to one of them, then $\delta (\textbf{x},\textbf{y})+\delta(\textbf{y},\textbf{z})=\delta(\textbf{x},\textbf{z})$. In such a case we say that  \textquotedblleft $\textbf{y}$ lies between $\textbf{x}$ and $\textbf{z}$\textquotedblright.

For any $\textbf{x}\in \Z^4$, we reduce $\textbf{x}$ by adding appropriate elements of $L_k$ until the resulting vector lies within $l^1$-distance $k$ of $\textbf{0}$ or some other element of $L_k$.
The first stage is to reduce $\textbf{x}$ to a vector whose coordinates all have absolute value at most $a+1$. If $\textbf{x}$ has a coordinate with absolute value above $a+1$, then let $\textbf{v}$ be one of the vectors $\pm \textbf{v}_i(1 \leq i \leq 8)$ such that the coordinates of  $\textbf{v}$ have the same sign as the corresponding coordinates of $\textbf{x}$. If a coordinate of $\textbf{x}$ is 0 then either sign is allowed for $\textbf{v}$ as long as the corresponding coordinate of $\textbf{v}$ has absolute value $\leq a+1$. So in the case $k \equiv 0 \pmod 2$ if the $\textbf{e}_3$ coordinate of $\textbf{x}$ is 0 then we  avoid $\textbf{v}_7$ and take $\textbf{v}_5$ instead. Also if the $\textbf{e}_4$ coordinate of $\textbf{x}$ is 0 (or both $\textbf{e}_3$ and $\textbf{e}_4$ coordinates are 0) then instead of $\textbf{v}_5$ we take $\textbf{v}_1$.

Now consider $\textbf{x}'=\textbf{x}-\textbf{v}$. If a coordinate of $\textbf{x}$ has absolute value $s, 1\leq s\leq a+1$, then the corresponding coordinate of $\textbf{x}'$ will have absolute value $s'\leq a+1$ because of the sign matching and the fact that the coordinates of $\textbf{v}$ have absolute value $\leq a+2$. If a coordinate of $\textbf{x}$ has absolute value $s=0$, then as indicated above, the corresponding value of $\textbf{x}'$ will have absolute value $s' \leq a+1$ because $\textbf{v}$ is chosen such that the corresponding coordinate has absolute value $\leq a+1$. If a coordinate of $\textbf{x}$ has absolute value $s>a+1$, then the corresponding coordinate of $\textbf{x}'$ will be strictly smaller in absolute value. Therefore repeating this procedure will result in a vector whose coordinates all have absolute value at most $a+1$.

If the resulting vector $\textbf{x}'$ lies between $\textbf{0}$ and $\textbf{v}$, where $\textbf{v}=\pm\textbf{v}_i$ for some $i$, then we have $\delta(\textbf{0},\textbf{x}')+\delta(\textbf{x}',\textbf{v})=\delta(\textbf{0},\textbf{v})$. For $k\equiv 0 \pmod 2$ all of the vectors $\textbf{v}$ satisfy $\delta(\textbf{0},\textbf{v})=4a+1$, and for $k\equiv 1 \pmod 2$ they all satisfy $\delta(\textbf{0},\textbf{v})=4a-1$. So in either case we have $\delta(\textbf{0},\textbf{v})=2k+1$. Since $\delta(\textbf{0},\textbf{x}')$ and $\delta(\textbf{x}',\textbf{v})$ are both non-negative integers, one of them must be at most $k$, so that $\textbf{x}' \in S_{f,k}+L_k$. Hence we also have $\textbf{x} \in S_{f,k}+L_k$ as required.

Now we are left with the case where the absolute value of each coordinate of the reduced $\textbf{x}$ is at most $a+1$, and $\textbf{x}$ is in the orthant of $\textbf{v}$, where $\textbf{v} = \pm \textbf{v}_i$ for some $i \leq 8$ but does not lie between $\textbf{0}$ and $\textbf{v}$.
Since $L_k$ is centrosymmetric we only need to consider the eight orthants containing $\textbf{v}_1, ..., \textbf{v}_8$.
For both cases $k \equiv0$ and $k\equiv1 \pmod 2$ the exceptions need to be considered for each orthant in turn. To avoid this paper being unduly long only the exceptions for the orthant of $\textbf{v}_1$ for $k \equiv 0$ (mod 2) and for $k \equiv 1$ (mod 2) are included here. The other orthants are handled similarly. A full proof including all orthants for both cases is available on ArXiv \cite {Lewis}.

So suppose that $k \equiv 0 \pmod 2$ and $\textbf{x}$ lies within the orthant of $\textbf{v}_1$, but not between $\textbf{0}$ and $\textbf{v}_1$. Then as $\textbf{v}_1 = (-a-1,a+1,a,-a+1)$, the third coordinate of $\textbf{x}$ is equal to $a+1$ or the fourth coordinate equals $-a$ or $-a-1$. We now distinguish three cases.

Case 1: $\textbf{x}=(-r,s,a+1,-u)$ where $0\leq r, s\leq a+1$ and $a \leq u \leq a+1$.
Let $\textbf{x}' = \textbf{x} - \textbf{v}_1 = (a+1-r, s-a-1, 1, a-1-u)$, which lies between $\textbf{0}$ and $-\textbf{v}_7$ unless $r \leq 1$ or $s \leq 1$. Let $\textbf{x}'' = \textbf{x}'+\textbf{v}_7=(2-r,s-2,-a-1,2a-u)$.
If $r \leq 1$ and $s \leq 1$ then $\textbf{x}''$ lies between $\textbf{0}$ and $-\textbf{v}_1$ unless $u=a$, in which case let $\textbf{x}''' = \textbf{x}''+ \textbf{v}_1 = (1-a-r, a-1+s, -1, a+1-u)$ which lies between $\textbf{0}$ and $\textbf{v}_7$.
If $r \leq 1$ and $s \geq 2$ then $\textbf{x}''$ lies between $\textbf{0}$ and $-\textbf{v}_3$.
If $r \geq 2$ and $s \leq 1$ then $\textbf{x}''$ lies between $\textbf{0}$ and $-\textbf{v}_2$.

Case 2: $\textbf{x}=(-r,s,a+1,-u)$ where $0\leq r, s\leq a+1$ and $0 \leq u \leq a-1$.
Let $\textbf{x}' = \textbf{x} - \textbf{v}_1 = (a+1-r, s-a-1, 1, a-1-u)$, which lies between $\textbf{0}$ and $-\textbf{v}_6$ unless $r = 0$ or $s = 0$.
Let $\textbf{x}'' = \textbf{x}'+\textbf{v}_6=(1-r,s-1,-a,-u-1)$. If $r = 0$ and $s = 0$ then $\textbf{x}''$ lies between $\textbf{0}$ and $-\textbf{v}_5$.
If $r = 0$ and $s \geq 1$ then $\textbf{x}''$ lies between $\textbf{0}$ and $-\textbf{v}_4$.
If $r \geq 1$ and $s = 0$ then $\textbf{x}''$ lies between $\textbf{0}$ and $-\textbf{v}_8$.

Case 3: $\textbf{x}=(-r,s,t,-u)$ where $0\leq r, s\leq a+1$ and $0 \leq t \leq a$ and $a \leq u \leq a+1$.
Let $\textbf{x}' = \textbf{x} - \textbf{v}_1 = (a+1-r, s-a-1, t-a, a-1-u)$, which lies between $\textbf{0}$ and $-\textbf{v}_5$ unless $r = 0$ or $s = 0$ or $t=0$. If $r=0$ and $s=0$, then $\textbf{x}$ lies between $\textbf{0}$ and $-\textbf{v}_7$.
Let $\textbf{x}'' = \textbf{x}'+\textbf{v}_5=(1-r,s-1,t-1,2a+1-u)$. If $r = 0, s\geq 1$ and $t \geq 1$ then $\textbf{x}''$ lies between $\textbf{0}$ and $\textbf{v}_8$.
Let $\textbf{x}'''=\textbf{x}+\textbf{v}_4 = (-a-r,s-a,a+t,a+1-u)$.
If $r = 0$ and $s \geq 1$ and $t=0$, then $\textbf{x}'''$ lies between $\textbf{0}$ and $\textbf{v}_4$ unless $s=a+1$, in which case if $u=a$ then $\textbf{x}$ lies between $\textbf{0}$ and $\textbf{v}_2$, and if $u=a+1$ then $\textbf{x}'''$ lies between $\textbf{0}$ and $\textbf{v}_4$.
Let $\textbf{x}'''' = \textbf{x}-\textbf{v}_3=(a+1-r,a-1+s,t-a-1,a-u)$.
If $r \geq 1, s= 0$ and $t \geq 1$ then $\textbf{x}''''$ lies between $\textbf{0}$ and $-\textbf{v}_4$. If $r \geq 1, s= 0$ and $t = 0$ then $\textbf{x}''''$ lies between $\textbf{0}$ and $-\textbf{v}_3$ if $u=a$, and between $\textbf{0}$ and $\textbf{v}_6$ if $u=a+1$.
If $r \geq 1, s \geq 1$ and $t = 0$ then $\textbf{x}''$ lies between $\textbf{0}$ and $\textbf{v}_7$ unless $r=a+1$ or $s=a+1$.
If $r = a+1$, $s \geq 1$ and $t =0$ then $\textbf{x}'$ lies between $\textbf{0}$ and $-\textbf{v}_8$.
If $r \geq 1$, $s=a+1$ and $t =0$ then $\textbf{x}'$ lies between $\textbf{0}$ and $-\textbf{v}_4$.

This completes the cases for the orthant of $\textbf{v}_1$ for $k \equiv 0 \pmod 2$.

Now suppose that $k \equiv 1 \pmod 2$ and $\textbf{x}$ lies within the orthant of $\textbf{v}_1$, but not between $\textbf{0}$ and $\textbf{v}_1$. Then the first coordinate of $\textbf{x}$ is equal to $-a$ or $-a-1$, or the third coordinate equals $-a$ or $-a-1$, or the fourth equals $a+1$. We distinguish seven cases.

Case 1: $\textbf{x}=(-r,s,-t,a+1)$ where $a \leq r, t \leq a+1$ and $0 \leq s \leq a+1$.
Let $\textbf{x}' = \textbf{x} - \textbf{v}_1 = (a-1-r, s-a-1, a-1-t,1)$, which lies between $\textbf{0}$ and $\textbf{v}_8$ unless $s \leq 1$ in which case let $\textbf{x}'' = \textbf{x}'-\textbf{v}_8=(2a-2-r, s-2,2a-1-t,-a)$ which lies between $\textbf{0}$ and $-\textbf{v}_1$.

Case 2: $\textbf{x}=(-r,s,-t,u)$ where $a \leq r, t\leq a+1$ and $0 \leq s \leq a+1$ and $0 \leq u \leq a$.
Let $\textbf{x}' = \textbf{x} - \textbf{v}_1 = (a-1-r, s-a-1, a-1-t, u-a)$, which lies between $\textbf{0}$ and $\textbf{v}_5$ unless $s = 0$ or $u \leq 1$, in which case let $\textbf{x}'' = \textbf{x}'-\textbf{v}_5=(2a-1-r,s-1,2a-t,u-2)$. If $s = 0$ and $u \leq 1$ then $\textbf{x}''$ lies between $\textbf{0}$ and $-\textbf{v}_1$, unless $t=a$, in which case let $\textbf{x}'''=\textbf{x}''+\textbf{v}_1=(a-r,a,1,u+a-2)$ which lies between $\textbf{0}$ and $\textbf{v}_4$.
If $s = 0$ and $u \geq 2$ then $\textbf{x}''$ lies between $\textbf{0}$ and $-\textbf{v}_7$. If $s \geq 1$ and $u \leq 1$ then $\textbf{x}''$ lies between $\textbf{0}$ and $-\textbf{v}_8$ unless $s=a+1$, in which case let $\textbf{x}''''=\textbf{x}''+\textbf{v}_8=(a-r,1,a-t,a+u-1)$ which lies between $\textbf{0}$ and $\textbf{v}_1$.

Case 3: $\textbf{x}=(-r,s,-t,a+1)$ where $a\leq r\leq a+1$, $0\leq s \leq a+1$ and $0 \leq t \leq a-1$.
Let $\textbf{x}' = \textbf{x} - \textbf{v}_1 = (a-1-r, s-a-1, a-1-t,1)$, which lies between $\textbf{0}$ and $-\textbf{v}_6$ unless $s = 0$, in which case let $\textbf{x}'' = \textbf{x}'+\textbf{v}_6 = (2a-1-r,-1,-t,-a+1)$  which lies between $\textbf{0}$ and $-\textbf{v}_4$.

Case 4: $\textbf{x}=(-r,s,-t,a+1)$ where $0\leq r\leq a-1$, $0\leq s \leq a+1$ and $a \leq t \leq a+1$.
Let $\textbf{x}' = \textbf{x} - \textbf{v}_1 = (a-1-r, s-a-1, a-1-t,1)$, which lies between $\textbf{0}$ and $-\textbf{v}_3$ unless $s \leq 1$, in which case let $\textbf{x}'' = \textbf{x}'+\textbf{v}_3 = (-2-r,s-2,2a-2-t,-a+1)$  which lies between $\textbf{0}$ and $-\textbf{v}_2$.

Case 5: $\textbf{x}=(-r,s,-t,a+1)$ where $0\leq r,t\leq a-1$ and $0 \leq s \leq a+1$.
Let $\textbf{x}' = \textbf{x} - \textbf{v}_1 = (a-1-r, s-a-1, a-1-t,1)$, which lies between $\textbf{0}$ and $-\textbf{v}_7$ unless $s =0$, in which case let $\textbf{x}'' = \textbf{x}'+\textbf{v}_7 = (-r-1,-1,-t-1,-a+2)$  which lies between $\textbf{0}$ and $\textbf{v}_5$.

Case 6: $\textbf{x}=(-r,s,-t,u)$ where $0\leq r\leq a-1$, $0\leq s \leq a+1$, $a \leq t \leq a+1$ and $0 \leq u \leq a$.
Let $\textbf{x}' = \textbf{x} - \textbf{v}_1 = (a-1-r, s-a-1, a-1-t,u-a)$, which lies between $\textbf{0}$ and $-\textbf{v}_4$ unless $s = 0$ or $u=0$, in which case let $\textbf{x}'' = \textbf{x}'+\textbf{v}_4 = (-r-1,s-1,2a-1-t,u-1)$. If $s=0$ and $u=0$ then let $\textbf{x}''' = \textbf{x}''+\textbf{v}_2=(a-r,a,a+1-t,a-2)$ which lies between $\textbf{0}$ and $-\textbf{v}_5$.
If $s = 0$ and $u \geq 1$ then $\textbf{x}''$ lies between $\textbf{0}$ and $-\textbf{v}_6$.
If $s \geq 1$ and $u =0$ then $\textbf{x}''$ lies between $\textbf{0}$ and $\textbf{v}_3$ unless $s=a+1$, in which case $\textbf{x}'$ lies between $\textbf{0}$ and $\textbf{v}_6$.

Case 7: $\textbf{x}=(-r,s,-t,u)$ where $a \leq r\leq a+1$, $0\leq s \leq a+1$, $0 \leq t \leq a-1$ and $0 \leq u \leq a$.
Let $\textbf{x}' = \textbf{x} - \textbf{v}_1 = (a-1-r, s-a-1, a-1-t,u-a)$, which lies between $\textbf{0}$ and $-\textbf{v}_2$ unless $t = 0$ or $u=0$, in which case let $\textbf{x}'' = \textbf{x}'+\textbf{v}_2 = (2a-r,s,-t+1,u-1)$. If $t=0$ and $u=0$ then let $\textbf{x}''' = \textbf{x}''+\textbf{v}_8=(a+1-r,s-a+1,-a+1,a)$ which lies between $\textbf{0}$ and $-\textbf{v}_3$ unless $a \leq s \leq a+1$, in which case let $\textbf{x}'''' = \textbf{x}-\textbf{v}_7=(a-r,s-a,a,a-1)$ which lies between $\textbf{0}$ and $\textbf{v}_4$.
If $t = 0$ and $u \geq 1$ then $\textbf{x}''$ lies between $\textbf{0}$ and $-\textbf{v}_5$ unless $s=a+1$ or $u=a$ in which case let $\textbf{x}^v = \textbf{x}''+\textbf{v}_5=(a-r,s-a,-a,-a+u+1)$. 
If $s = a+1$ then $\textbf{x}'$ lies between $\textbf{0}$ and $\textbf{v}_3$. If $1 \leq s \leq a$ and $u=a$ then $\textbf{x}^v$ lies between $\textbf{0}$ and $\textbf{v}_8$.
If $s = 0$ and $u=a$ then $\textbf{x}''$ lies between $\textbf{0}$ and $-\textbf{v}_7$. If $t \geq 1$ and $u=0$ then $\textbf{x}''$ lies between $\textbf{0}$ and $\textbf{v}_6$ unless $s=a+1$, in which case  $\textbf{x}'$ lies between $\textbf{0}$ and $\textbf{v}_3$.

This completes the cases for the orthant of $\textbf{v}_1$ for $k \equiv 1 \pmod 2$. 
\end{proof}


\section {Conclusion}

This paper has reviewed the proven extremal circulant graphs of dimension 1 and 2, degree 2 to 5, and Dougherty and Faber's largest known solutions for dimension 3, degree 6 and 7, in the context of relevant lower and upper bounds. It has also presented newly discovered families of circulant graphs of dimension 4, degree 8 and 9, which are proven to be extremal for small diameters above a threshold and are conjectured to remain so for all diameters, along with a proof of the existence of the degree 8 graph for all diameters.

A question naturally arises whether there are any common relationships that are valid across the graphs of all four dimensions.
The formulae for the order of the extremal and largest known graphs of degree $d \leq 9$ are compared with the lower bounds $CJ(d,k)$. We recall that $CJ(d,k)=(1/2)(4/f)^fk^f+O(k^{f-1})$ for even $d$ and $f=d/2$, and observe that for degree $d=2,4,6$ and 8 the coefficients of the leading term $k^f$ in the formulae for graph order are $2,2,32/27$ and $1/2$ respectively, and therefore equal to that of $CJ(d,k)$ in each case, see Table \ref{table:7A}.

\begin{table}[!htbp]
\small
\caption{\small Coefficients of the two leading terms in the formulae for the order of extremal and largest known graphs of degree $d \leq 9$} 
\centering 
\begin{tabularx} {\linewidth} {c c l  * {3} {@ { } c}} 
\noalign {\vskip 2mm} 
\hline\hline 
\noalign {\vskip 1mm} 
Degree & Dimension & order & Coefficient & Coefficient & Coefficient \\
$d$ & $f$ & $ CC/DF/L(d,k)$ & of $k^f$ & of $k^{f-1}$ & of $k^f$ in $CJ(d,k)$ \\ 
\hline 
\noalign {\vskip 1mm} 
2 & 1 & $2k+1$ & 2 & 1 & 2 \\ 
3 & 1 & $4k$ & 4 & 0 & - \\
4 & 2 & $2k^2+2k+1$ & 2 & 2 & 2 \\
5 & 2 & $4k^2$ & 4 & 0 & - \\
6 & 3 & $(32k^3+48k^2)/27+O(k)$ & $32/27$ & $48/27$ & $32/27$ \\
7 & 3 & $64k^3/27+O(k)$ & $64/27$ & 0 & - \\
8 & 4 & $(k^4+2k^3)/2+O(k^2)$ & $1/2$ & 1 & $1/2$ \\
9 & 4 & $k^4+O(k^2)$ & 1 & 0 & - \\ [1ex] 
\hline 
\end{tabularx}
\label{table:7A} 
\end{table}

This supports a conjecture that for any even degree $d$, the leading term in the formula for the order of an extremal graph is $(1/2)(4/f)^fk^f$ where $f=d/2$. We may also observe from Table \ref {table:7A} that the second term in all four cases is equal to $(4/f)^{f-1}k^{f-1}$. For odd degree $d$ we see that the leading coefficient is double the coefficient for the even degree of the same dimension in all four cases. Also the coefficient of $k^{f-1}$ in each case is zero. This is summarised in Table \ref{table:7B}.

\begin{table}[!htbp]
\small
\caption{\small Expressions for the first two terms in the formulae for the order of largest known graphs of degree $d \leq 9$} 
\centering 
\begin{tabularx} {\linewidth} {L L L L} 
\noalign {\vskip 2mm} 
\hline\hline 
\noalign {\vskip 1mm} 
Degree $d$ & Dimension $f$  & Leading Term & Second Term \\
\hline 
\noalign {\vskip 2mm}  
$d=2,4,6$ or 8 & $f=d/2$ & $ \frac {1}{2} \left ( \frac {4}{f} \right )^f k^f$ & $ \left ( \frac {4}{f} \right )^{f-1} k^{f-1}$ \\ 
\noalign {\vskip 2mm}  
$d=3,5,7$ or 9 & $f=(d-1)/2$ & $ \left ( \frac {4}{f} \right )^f k^f$ & 0 \\  [1ex] 
\noalign {\vskip 1mm}  
\hline 
\end{tabularx}
\label{table:7B} 
\end{table}

If these terms were to remain valid for degree $d \geq 10$ then for dimension 5 and 6 this would give the following formulae for extremal circulant graph order:
\[ n(d,k) =
\begin {cases}
(512k^5+1280k^4)/3125 + O(k^3) & \mbox { for } d=10 \\
1024k^5/3125 + O(k^3) & \mbox { for } d=11 \\
(32k^6+96k^5)/729 + O(k^4) & \mbox { for } d=12 \\
64k^6/729 + O(k^4) & \mbox { for } d=13 \\
\end {cases}
\]
It is hoped that the results in this paper will provide a useful framework for further research within this area to identify new families of circulant graphs, for degree $d\geq 10$ and arbitrary diameter $k$, that can be proven extremal for some range of $k$.

%

\section {Acknowledgements}

I would like to thank Professor Jozef Siran for reviewing the draft and making many helpful comments to improve the presentation.


\end{document}